\documentclass[11pt]{amsart}
\usepackage{amsmath,amssymb,amsthm,mathtools}
\usepackage{mathrsfs}
\usepackage{enumitem}
\usepackage[colorlinks=true,linkcolor=blue,citecolor=blue,urlcolor=blue]{hyperref}
\numberwithin{equation}{section}
\newtheorem{theorem}{Theorem}[section]
\newtheorem{proposition}[theorem]{Proposition}
\newtheorem{lemma}[theorem]{Lemma}
\newtheorem{corollary}[theorem]{Corollary}

\theoremstyle{definition}
\newtheorem{definition}[theorem]{Definition}
\newtheorem{notation}[theorem]{Notation}
\newtheorem{example}[theorem]{Example}

\theoremstyle{remark}
\newtheorem{remark}[theorem]{Remark}

\newcommand{\HH}{\mathbb H}
\newcommand{\CC}{\mathbb C}
\newcommand{\RR}{\mathbb R}
\newcommand{\PP}{\mathbb P}
\newcommand{\QQ}{\mathbb Q}
\newcommand{\ZZ}{\mathbb Z}
\newcommand{\OO}{\mathcal O}
\newcommand{\HHc}{\mathbb H_{\mathbb C}}
\newcommand{\Scal}{\mathbb S}
\newcommand{\DeltaA}{\Delta_A}
\newcommand{\St}{\operatorname{St}}
\newcommand{\Zar}{\mathrm{Zar}}
\newcommand{\ord}{\operatorname{ord}}
\newcommand{\CritVal}{\operatorname{CritVal}}
\newcommand{\Newt}{\operatorname{Newt}}
\newcommand{\rtor}{r_{\mathrm{tor}}}

\title[Four-Point Picard Theorem]{The Four-Point Picard Theorem for Quaternionic Slice Regular Functions}
\author[G. Ren]{Guangbin Ren}
\address{Department of Mathematics, University of Science and Technology of China, Hefei 230026, China}
\email{rengb@ustc.edu.cn}
\author[X. Zhao]{Xin Zhao}
\address{Department of Mathematics, University of Science and Technology of China, Hefei 230026, China}
\email{zx130781@mail.ustc.edu.cn}
\thanks{Corresponding author: Xin Zhao.}
\date{}
\hypersetup{pdftitle={The Four-Point Picard Theorem for Quaternionic Slice Regular Functions},pdfauthor={Guangbin Ren and Xin Zhao},pdfsubject={Four-point Picard theorem for quaternionic slice regular functions},pdfkeywords={quaternionic slice regular functions, Picard theorem, stem functions, quaternionic analysis, semi-abelian varieties, Nevanlinna theory, value distribution}}

\begin{document}

\begin{abstract}
An entire slice regular function $f:\HH\to\HH$ can omit four prescribed quaternionic values only in the affine-dependent case.  More precisely, four affinely independent omitted values force $f$ to be constant, while the converse follows from the plane-omission theorem of Bisi--Winkelmann.  The proof passes to the real-symmetric stem function.  For each omitted value a quadratic zero-divisor criterion gives a zero-free entire function $Q_j$, and the component normal to the affine span is governed by a square-discriminant identity.  Finite-order data are excluded by Hadamard factorization and a rigidity argument on the real axis.  In the general case, logarithmic Bloch--Ochiai places the $Q$-curve in a translated algebraic torus.  The Laurent-square case reduces to the finite-order contradiction, and the nonsquare case is excluded by an even-ramification argument together with the level-one truncated Second Main Theorem of Noguchi--Winkelmann--Yamanoi.
\end{abstract}

\maketitle

\section{Introduction}
\label{sec:introduction}

Quaternionic slice regular Picard phenomena are governed by the real affine geometry of omitted sets, rather than by isolated values alone.  The Bisi--Winkelmann results needed here are the following.

\begin{theorem}[Bisi--Winkelmann]\label{thm:bw-inputs}
The following hold.
\begin{enumerate}[label=\textup{(\roman*)}]
\item For every real affine $2$-plane $P\subset\HH\simeq\RR^4$, there exists a nonconstant entire slice regular function $f:\HH\to\HH$ such that
\[
f(\HH)=\HH\setminus P.
\]
\item If $q_1,\ldots,q_5\in\HH$ are not contained in any real affine $3$-space, then every entire slice regular function omitting $q_1,\ldots,q_5$ is constant.
\end{enumerate}
\end{theorem}
The plane-omission assertion follows from \cite[Proposition~5.1, Proposition~6.1 and the subsequent remark]{BisiWinkelmann2020QuaternionicPicard}; the five-point assertion is \cite[Theorem~3.5]{BisiWinkelmann2020QuaternionicPicard}.

Part (i) of Theorem~\ref{thm:bw-inputs} gives the sharpness direction in Corollary~\ref{cor:exact-four-value}.  The unresolved finite configuration is a quadruple spanning a real affine $3$-space; the theorem below proves its rigidity.

The proof is written in the stem-function formalism for slice regular functions, following Cullen, Gentili--Struppa, and subsequent work \cite{Cullen1965Integral,GentiliStruppa2006Cullen,GentiliStruppa2007NewTheory,GentiliStoppato2008Zeros,GentiliStoppatoStruppa2013Regular,ColomboGentiliSabadiniStruppa2009Extension,ColomboSabadiniStruppa2011FunctionalCalculus,GhiloniPerotti2011Alternative}.  Jensen-type formulae and zero-counting results for slice regular functions provide related value-distribution background; see \cite{Perotti2020Jensen,BisiWinkelmann2020Harmonicity}.  The main result is the following.

\begin{theorem}[Four-point Picard theorem]\label{thm:main}
Let $a_0,a_1,a_2,a_3\in\HH$ be affinely independent. If
\[
f:\HH\to\HH
\]
is an entire slice regular function omitting $a_0,a_1,a_2,a_3$, then $f$ is constant.
\end{theorem}

Together with the Bisi--Winkelmann plane-omission construction, this gives the exact four-value criterion.

\begin{corollary}[Exact four-value omission criterion]\label{cor:exact-four-value}
Let $a_0,a_1,a_2,a_3\in\HH$. There exists a nonconstant entire slice regular function omitting all four values if and only if the four points are affinely dependent.
\end{corollary}

Theorem~\ref{thm:main} is not a formal consequence of the five-point theorem, since four affinely independent points always lie in a real affine $3$-space.  The obstruction is encoded by a single normal coordinate after complexifying this affine span.  Let
\[
F:\CC\to\HHc
\]
be the real-symmetric stem function associated with $f$.  For $c\in\HH$ set
\[
Q_c(z)=\langle F(z)-c,F(z)-c\rangle,
\]
where the form is the complex bilinear extension of the Euclidean inner product on $\HH$.  The Bisi--Winkelmann zero-divisor criterion identifies omission of $c$ with zero-freeness of $Q_c$.  Thus the four omitted values give four zero-free entire functions of one complex variable.

After translating the values, put $p_0=0$ and $p_j=a_j-a_0$ for $1\le j\le3$.  Let
\[
V=\operatorname{span}_{\RR}\{p_1,p_2,p_3\},\qquad \HHc=V_{\CC}\oplus\CC n,
\]
with $\|n\|=1$, and write $F=W+\tau n$.  For
\[
Q_j(z)=\langle F(z)-p_j,F(z)-p_j\rangle\qquad(0\le j\le3)
\]
one has
\[
\langle W,p_j\rangle=\frac{Q_0+\|p_j\|^2-Q_j}{2}=:b_j(Q).
\]
If $G=(\langle p_i,p_j\rangle)_{1\le i,j\le3}$ is the Gram matrix, then
\[
W(Q)=\sum_{i,k=1}^3(G^{-1})_{ik}b_k(Q)p_i,
\qquad
\langle W,W\rangle=b(Q)^TG^{-1}b(Q).
\]
Consequently
\[
\tau^2=Q_0-b(Q)^TG^{-1}b(Q)=:\DeltaA(Q_0,Q_1,Q_2,Q_3).
\]
Section~\ref{sec:discriminant} proves the converse: four-value omission is equivalent to the existence of real-symmetric zero-free functions $Q_0,\ldots,Q_3$ for which $\DeltaA(Q)$ has a real-symmetric entire square root.  Proposition~\ref{prop:normal-form} records the corresponding normal-square coordinate form.

Affine independence is used through the invertibility of the Gram matrix: it makes the tangential component $W$ a linear function of the four quadratic data and leaves the normal square as the remaining constraint.  If the four points are affinely dependent, this reconstruction degenerates and the normal coordinate no longer yields a four-point obstruction; the plane-omission part of Theorem~\ref{thm:bw-inputs} then supplies nonconstant examples.  The five-point theorem of Bisi--Winkelmann rests on an algebraic constraint among five quadratic functions.  By contrast, Proposition~\ref{prop:normal-form} shows that no such algebraic constraint exists for four affinely independent points; the residual condition is the square-discriminant relation.

Finite-order $Q$-data are handled by one-variable methods.  In that case $Q_j=e^{P_j}$ with $P_j\in\RR[z]$, by the zero-free case of the Hadamard factorization theorem \cite[Ch.~I, \S\S~8--9]{Boas1954EntireFunctions,Levin1980Zeros}.  Since
\[
Q_j(x)=\|F(x)-p_j\|^2\qquad (x\in\RR),
\]
a real-axis distance estimate leaves only two asymptotic possibilities for the polynomials $P_j$.  Each alternative makes the square-discriminant identity force either $e^P-C$ or $1-Ce^P$ to be an entire square, with $P$ nonconstant and $C\ne0$, which is impossible.  The finite-order statement used later for auxiliary $Q$-data contains the polynomial and finite-order cases stated as Theorems~\ref{thm:polynomial-case} and~\ref{thm:finite-order}.

For unrestricted growth, let
\[
Q=(Q_0,Q_1,Q_2,Q_3):\CC\to(\CC^*)^4.
\]
Logarithmic Bloch--Ochiai gives the Zariski closure of $Q(\CC)$ as a translate of an algebraic subtorus; the form used here is Noguchi's semi-abelian theorem \cite[Main Theorem~(i)]{Noguchi1998HolomorphicCurves}, together with the formulations in \cite{NoguchiWinkelmannYamanoi2002SMT,DethloffLu2001LogarithmicJetBundles}.  Positivity of $Q_j(x)$ on the real axis allows the translate to be chosen positive real.  In torus coordinates,
\[
Q_j=\lambda_jE^{m_j},\qquad \lambda_j>0,
\]
with $E:\CC\to(\CC^*)^d$ Zariski dense.  The discriminant becomes a Laurent polynomial
\[
D_M(u)=\DeltaA(\lambda_0u^{m_0},\lambda_1u^{m_1},\lambda_2u^{m_2},\lambda_3u^{m_3}),
\]
and $D_M(E)=\tau^2$.  The Laurent-square case is reduced to the finite-order contradiction by pulling back along finite-order dense exponential curves in the torus.

For the nonsquare Laurent case, let $B_D$ denote the squarefree branch part of a Laurent polynomial $D$.  From $D(E)=\tau^2$ every zero of $B_D(E)$ has even order.  Quotienting by the connected stabilizer of the branch divisor gives a Zariski-dense curve in a torus and a reduced divisor with trivial stabilizer whose pullback has multiplicity at least two everywhere.  The level-one truncated Second Main Theorem of Noguchi--Winkelmann--Yamanoi gives
\[
T_G(r;L(\overline{\mathcal D}))\le N_1(r;G^*\mathcal D)+\varepsilon T_G(r;L(\overline{\mathcal D}))
\]
for reduced divisors on semi-abelian varieties, in the notation of Section~\ref{sec:laurent-ramification}, by \cite[Main Theorem~(iii)]{NoguchiWinkelmannYamanoi2008SMTII}.  The ramification hypothesis gives $N_1\le \frac12N$, and the First Main Theorem gives $N\le T_G+O(1)$.  For $\varepsilon<1/2$ the characteristic is bounded.  Trivial connected stabilizer makes the relevant Newton polytope full-dimensional, so the divisor closure is big on a smooth toric compactification; bounded characteristic with respect to a big line bundle contradicts Zariski density.  Positive toric rank is therefore impossible.  This proves Theorem~\ref{thm:main}.  For background on Nevanlinna theory for semi-abelian varieties, see \cite{GriffithsKing1973Nevanlinna,NoguchiWinkelmannYamanoi2007Degeneracy,NoguchiWinkelmann2014Nevanlinna,Noguchi1981LogDerivatives}.

The paper is organized as follows.  Sections~\ref{sec:stem} and~\ref{sec:discriminant} establish the zero-divisor criterion and the discriminant reduction.  Section~\ref{sec:finite-order} treats polynomial and finite-order data.  Section~\ref{sec:toroidal} gives the toroidal reduction, Section~\ref{sec:laurent-ramification} proves the Laurent ramification obstruction, and Section~\ref{sec:proof} completes the proof.  Appendix~\ref{app:technical} contains the toric and Nevanlinna lemmas used in Sections~\ref{sec:toroidal} and~\ref{sec:laurent-ramification}; Appendix~\ref{app:external-inputs} lists the external results.

\section{Stem functions and quadratic zero divisors}
\label{sec:stem}

\subsection{Complexified quaternions}

Let
\[
\HH=\RR+\RR\mathbf i+\RR\mathbf j+\RR\mathbf k
\]
be the quaternion algebra, and let
\[
\HHc:=\HH\otimes_{\RR}\CC
\]
be its complexification. The scalar imaginary unit of $\CC$ is denoted by $\mathrm i$, to distinguish it from the quaternionic unit $\mathbf i$.

Every element $Z\in\HHc$ can be written uniquely as
\[
Z=z_0+z_1\mathbf i+z_2\mathbf j+z_3\mathbf k,
\qquad z_\ell\in\CC.
\]
We extend the Euclidean inner product on $\HH\simeq\RR^4$ complex bilinearly:
\[
\langle Z,W\rangle=\sum_{\ell=0}^3z_\ell w_\ell.
\]
This form is complex bilinear, not Hermitian. Its associated quadratic form is
\[
\langle Z,Z\rangle=z_0^2+z_1^2+z_2^2+z_3^2.
\]

Let
\[
\sigma:\HHc\to\HHc
\]
denote the anti-linear involution induced by complex conjugation on the scalar factor and fixing $\HH$. Thus
\[
\sigma\left(\sum_{\ell=0}^3 z_\ell e_\ell\right)=\sum_{\ell=0}^3\overline{z_\ell}e_\ell,
\]
where $e_0=1$, $e_1=\mathbf i$, $e_2=\mathbf j$, $e_3=\mathbf k$.

We write $\langle\cdot,\cdot\rangle_{\RR}$ for the Euclidean inner product on $\HH$, and $\langle\cdot,\cdot\rangle$ for its complex-bilinear extension to $\HHc$.  Quaternionic conjugation is denoted by $q\mapsto\overline q$ for $q\in\HH$, while complex conjugation of the variable is denoted by $z\mapsto\overline z$; the involution $\sigma$ is the induced conjugation on $\HHc$.

\subsection{Stem functions}

Let
\[
F:\CC\to\HHc
\]
be an entire holomorphic map. The map $F$ is called real-symmetric if
\[
F(\overline z)=\sigma(F(z)).
\]
Equivalently, writing
\[
F(z)=F_1(z)+\mathrm i F_2(z),\qquad F_1(z),F_2(z)\in\HH,
\]
one has
\[
F_1(\overline z)=F_1(z),\qquad F_2(\overline z)=-F_2(z).
\]

Let
\[
\Scal=\{J\in\HH:J^2=-1\}
\]
be the sphere of imaginary units. A real-symmetric entire stem function $F=F_1+\mathrm iF_2$ induces an entire slice regular function
\[
f:\HH\to\HH
\]
by
\[
f(x+yJ)=F_1(x+\mathrm i y)+JF_2(x+\mathrm i y),\qquad J\in\Scal.
\]
Conversely, every entire slice regular function arises in this way from a unique real-symmetric entire stem function.  This stem-function viewpoint will be used throughout; see \cite{GhiloniPerotti2011Alternative,GentiliStoppatoStruppa2013Regular} for broader formulations.

\subsection{Quadratic omission functions}

Let $F:\CC\to\HHc$ be a real-symmetric entire stem function, and let $c\in\HH$.  Define
\[
Q_c(z):=\langle F(z)-c,F(z)-c\rangle.
\]
Then $Q_c$ is entire and real-symmetric:
\[
Q_c(\overline z)=\overline{Q_c(z)}.
\]
The Bisi--Winkelmann zero-divisor criterion, cf. \cite[Proposition~2.2]{BisiWinkelmann2020QuaternionicPicard}, is used in the following form.  In the standard slice-regular notation, $Q_c$ is the stem of the normal function $N(f-c)=(f-c)*(f-c)^c$; see \cite{GentiliStoppato2008Zeros,GentiliStoppatoStruppa2013Regular} for the zero structure of normal functions.

\begin{lemma}[Omitted values and quadratic zero divisors]\label{lem:zero-divisor}
Let $F=F_1+\mathrm iF_2$ be a real-symmetric entire stem function, and let $f$ be the induced entire slice regular function. For $c\in\HH$, the following are equivalent:
\begin{enumerate}[label=\textup{(\roman*)}]
\item $f$ omits $c$;
\item $Q_c(z)=\langle F(z)-c,F(z)-c\rangle$ has no zero in $\CC$.
\end{enumerate}
\end{lemma}

\begin{proof}
Fix $z=x+\mathrm i y$ and write
\[
F(z)=F_1(z)+\mathrm iF_2(z),\qquad F_1(z),F_2(z)\in\HH.
\]
The corresponding slice values over the sphere associated with $z$ are
\[
f(x+yJ)=F_1(z)+JF_2(z),\qquad J\in\Scal.
\]
The parametrization $(z,J)\leftrightarrow(\overline z,-J)$ gives the same quaternionic point, so the criterion may be checked on these spherical fibers.  When $z=x\in\RR$, real-symmetry gives $F_2(x)=0$, and the criterion reduces to $Q_c(x)=\|F_1(x)-c\|^2=0$ if and only if $f(x)=c$.  On a nondegenerate fiber, $f$ takes the value $c$ precisely when there exists $J\in\Scal$ such that
\[
F_1(z)+JF_2(z)=c.
\]
Set
\[
a:=F_1(z)-c,\qquad b:=F_2(z).
\]
The question is whether there is $J\in\Scal$ such that
\[
a+Jb=0.
\]
On the other hand,
\[
Q_c(z)=\langle a+\mathrm i b,a+\mathrm i b\rangle
=\|a\|^2-\|b\|^2+2\mathrm i\langle a,b\rangle_{\RR}.
\]
Therefore $Q_c(z)=0$ if and only if
\[
\|a\|=\|b\|,\qquad \langle a,b\rangle_{\RR}=0.
\]

Assume first that $Q_c(z)=0$. If $b=0$, then $a=0$, so $F_1(z)=c$, and $f$ takes the value $c$. If $b\neq0$, define
\[
J:=-ab^{-1}.
\]
Then
\[
|J|=\frac{|a|}{|b|}=1,
\]
and
\[
\operatorname{Re}J=-\operatorname{Re}(ab^{-1})=-\|b\|^{-2}\operatorname{Re}(a\overline b)=-\|b\|^{-2}\langle a,b\rangle_{\RR}=0.
\]
Hence $J\in\Scal$, and $Jb=-a$. Thus $F_1(z)+JF_2(z)=c$.

Conversely, suppose $F_1(z)+JF_2(z)=c$ for some $J\in\Scal$. Then $a=-Jb$. Left multiplication by $J$ is an orthogonal transformation of $\HH$, so $\|a\|=\|b\|$. Moreover,
\[
\langle a,b\rangle_{\RR}=-\langle Jb,b\rangle_{\RR}=0.
\]
Thus $Q_c(z)=0$.

$Q_c$ has a zero if and only if $f$ takes the value $c$, as claimed.
\end{proof}

\section{The four-point discriminant}
\label{sec:discriminant}

Let $a_0,a_1,a_2,a_3\in\HH$ be affinely independent. Translating the target by $-a_0$, we may assume that
\[
p_0=0,\qquad p_j=a_j-a_0\quad (j=1,2,3).
\]
Thus $p_1,p_2,p_3$ are real-linearly independent in $\HH\simeq\RR^4$.

Let
\[
V:=\operatorname{span}_{\RR}\{p_1,p_2,p_3\}.
\]
Then $V$ is a three-dimensional real subspace of $\HH$.  Choose a unit vector
\[
n\in V^\perp,\qquad \|n\|=1.
\]
Let $V_{\CC}:=V\otimes_{\RR}\CC$.  Then
\[
\HHc=V_{\CC}\oplus\CC n.
\]
Let $G=(G_{ij})_{1\le i,j\le3}$ be the Gram matrix
\[
G_{ij}:=\langle p_i,p_j\rangle.
\]
Since $p_1,p_2,p_3$ are real-linearly independent, $G$ is positive definite and invertible.

Let $F:\CC\to\HHc$ be a real-symmetric entire stem function. With respect to the decomposition above, write
\[
F=W+\tau n,
\]
where
\[
W:\CC\to V_{\CC},\qquad \tau:\CC\to\CC
\]
are entire and real-symmetric. For $j=0,1,2,3$, define
\[
Q_j(z):=\langle F(z)-p_j,F(z)-p_j\rangle.
\]

\begin{definition}[Four-point omission datum]\label{def:omission-datum}
A four-point omission datum is a pair
\[
(F,(p_0,p_1,p_2,p_3))
\]
where $F:\CC\to\HHc$ is a real-symmetric entire stem function, the points $p_0,p_1,p_2,p_3\in\HH$ are affinely independent and normalized by $p_0=0$, and each function
\[
Q_j=\langle F-p_j,F-p_j\rangle,
\qquad j=0,1,2,3,
\]
is zero-free on $\CC$.  The datum is called nonconstant if $F$ is nonconstant.
\end{definition}

\begin{remark}
For a four-point omission datum, the associated $Q$-data satisfy the square-discriminant identity of Proposition~\ref{prop:square-identity}; this is a consequence of the decomposition $F=W+\tau n$, not an additional hypothesis.
\end{remark}

\begin{lemma}[Gram reconstruction]\label{lem:gram}
For $j=1,2,3$, define
\[
b_j(Q):=\frac{Q_0+\|p_j\|^2-Q_j}{2}
\]
and let $b(Q)=(b_1(Q),b_2(Q),b_3(Q))^T$. Then
\[
\langle W,p_j\rangle=b_j(Q),\qquad j=1,2,3.
\]
Consequently,
\[
W=\sum_{i,k=1}^3(G^{-1})_{ik}b_k(Q)p_i,
\]
and
\[
\langle W,W\rangle=b(Q)^TG^{-1}b(Q).
\]
\end{lemma}

\begin{proof}
Since $p_j\in V$ and $n\perp V$,
\[
\langle F-p_j,F-p_j\rangle=\langle W-p_j,W-p_j\rangle+\tau^2.
\]
For $j=0$, this gives
\[
Q_0=\langle W,W\rangle+\tau^2.
\]
For $j=1,2,3$, expanding gives
\[
Q_j=\langle W,W\rangle-2\langle W,p_j\rangle+\|p_j\|^2+\tau^2.
\]
Subtracting the formula for $Q_0$ gives
\[
\langle W,p_j\rangle=\frac{Q_0+\|p_j\|^2-Q_j}{2}=b_j(Q).
\]
Now write $W=\sum_{i=1}^3u_ip_i$. Then
\[
b_j(Q)=\sum_{i=1}^3u_iG_{ij}.
\]
In matrix form, $b=Gu$, since $G$ is symmetric. Thus $u=G^{-1}b$, and hence
\[
W=\sum_{i,k=1}^3(G^{-1})_{ik}b_k(Q)p_i.
\]
Finally,
\[
\langle W,W\rangle=u^TGu=b^TG^{-1}b.
\]
\end{proof}

\begin{definition}[Four-point discriminant]
The four-point discriminant associated with the ordered quadruple $A=(a_0,a_1,a_2,a_3)$ is
\[
\DeltaA(Q_0,Q_1,Q_2,Q_3):=Q_0-b(Q)^TG^{-1}b(Q).
\]
\end{definition}

\begin{proposition}[Square-discriminant identity]\label{prop:square-identity}
Let $F=W+\tau n$ be a real-symmetric entire stem function and let
\[
Q_j=\langle F-p_j,F-p_j\rangle,\qquad j=0,1,2,3.
\]
Then
\begin{equation}\label{eq:discriminant-identity}
\DeltaA(Q_0,Q_1,Q_2,Q_3)=\tau^2.
\end{equation}
\end{proposition}

\begin{proof}
Since $F=W+\tau n$ and $W\perp n$,
\[
Q_0=\langle F,F\rangle=\langle W,W\rangle+\tau^2.
\]
By Lemma \ref{lem:gram}, $\langle W,W\rangle=b(Q)^TG^{-1}b(Q)$. Hence
\[
\begin{aligned}
\tau^2
&=Q_0-b(Q)^TG^{-1}b(Q)\\
&=\DeltaA(Q_0,Q_1,Q_2,Q_3).
\end{aligned}
\]
\end{proof}

Equation~\eqref{eq:discriminant-identity} will be called the square-discriminant identity.

Here and below $\OO(\CC)^*$ denotes the group of units of $\OO(\CC)$, equivalently the zero-free entire functions.

\begin{proposition}[Reconstruction from square-discriminant data]\label{prop:reconstruction}
Let
\[
Q_0,Q_1,Q_2,Q_3\in\OO(\CC)^*
\]
be real-symmetric zero-free entire functions. Suppose that there exists a real-symmetric entire function $\tau\in\OO(\CC)$ such that
\[
\DeltaA(Q_0,Q_1,Q_2,Q_3)=\tau^2.
\]
Define
\[
W(Q):=\sum_{i,k=1}^3(G^{-1})_{ik}b_k(Q)p_i
\]
and
\[
F:=W(Q)+\tau n.
\]
Then $F$ is a real-symmetric entire stem function and satisfies
\[
\langle F-p_j,F-p_j\rangle=Q_j,\qquad j=0,1,2,3.
\]
Consequently, the slice regular function induced by $F$ omits $p_0,p_1,p_2,p_3$.
\end{proposition}

\begin{proof}
The functions $b_j(Q)$ are real-symmetric because the $Q_j$ are real-symmetric and all coefficients entering the definition are real. Hence $W(Q)$ is a real-symmetric entire map into $V_{\CC}$. Since $\tau$ is real-symmetric and $n\in\HH$, $F=W(Q)+\tau n$ is a real-symmetric entire stem function.

First,
\[
\langle F,F\rangle=\langle W(Q),W(Q)\rangle+\tau^2=b(Q)^TG^{-1}b(Q)+\tau^2=Q_0
\]
by the square-discriminant assumption. Thus $\langle F-p_0,F-p_0\rangle=Q_0$.

For $j=1,2,3$,
\[
\langle F-p_j,F-p_j\rangle
=\langle F,F\rangle-2\langle W(Q),p_j\rangle+\|p_j\|^2.
\]
By construction $\langle W(Q),p_j\rangle=b_j(Q)$, so
\[
\langle F-p_j,F-p_j\rangle
=Q_0-2b_j(Q)+\|p_j\|^2=Q_j.
\]
Thus all four identities hold. Since each $Q_j$ is zero-free, Lemma \ref{lem:zero-divisor} implies that the induced slice regular function omits $p_j$ for all $j$.
\end{proof}

\begin{theorem}[Four-point discriminant reduction]\label{thm:discriminant-reduction}
Let $a_0,a_1,a_2,a_3\in\HH$ be affinely independent. After translating by $-a_0$, set $p_0=0$ and $p_j=a_j-a_0$ for $j=1,2,3$. Define $V,n,G,b(Q)$, and $\DeltaA$ as above.

There exists a nonconstant entire slice regular function $f:\HH\to\HH$ omitting $a_0,a_1,a_2,a_3$ if and only if there exist real-symmetric zero-free entire functions
\[
Q_0,Q_1,Q_2,Q_3\in\OO(\CC)^*
\]
and a real-symmetric entire function $\tau\in\OO(\CC)$ such that
\[
\DeltaA(Q_0,Q_1,Q_2,Q_3)=\tau^2,
\]
and the stem function
\[
F=W(Q)+\tau n
\]
is nonconstant.
\end{theorem}

\begin{proof}
Assume first that such a nonconstant slice regular function exists. Translate by $-a_0$ and let $F$ be the real-symmetric entire stem function inducing $f-a_0$. Decompose $F=W+\tau n$ and define $Q_j=\langle F-p_j,F-p_j\rangle$. By Lemma \ref{lem:zero-divisor}, each $Q_j$ is zero-free. Since $F$ is real-symmetric and $p_j\in\HH$, each $Q_j$ is real-symmetric. By Proposition \ref{prop:square-identity}, $\DeltaA(Q)=\tau^2$. Since $f$ is nonconstant, $F$ is nonconstant.

Conversely, suppose such $Q_j$ and $\tau$ are given. Define $F=W(Q)+\tau n$. By Proposition \ref{prop:reconstruction}, $F$ is a real-symmetric entire stem function and $\langle F-p_j,F-p_j\rangle=Q_j$ for all $j$. Since the $Q_j$ are zero-free, Lemma \ref{lem:zero-divisor} implies that the induced slice regular function omits $p_0,p_1,p_2,p_3$. Adding $a_0$ gives a function omitting $a_0,a_1,a_2,a_3$. It is nonconstant precisely when $F$ is nonconstant.
\end{proof}

\begin{remark}[Real structure of the square root]
The square root in Theorem~\ref{thm:discriminant-reduction} is part of the real stem data.  The assertion that $\DeltaA(Q)$ is a square in $\OO(\CC)$ alone is weaker: a square root need not be real-symmetric, and then $W(Q)+\tau n$ need not define a real-symmetric stem function.  Thus the condition is formulated with a real-symmetric root $\tau$.
\end{remark}

\subsection{Normal form}

The next coordinate form identifies the discriminant with the global normal-square coordinate in the four quadratic variables.

Define
\[
R_j(W):=\langle W-p_j,W-p_j\rangle,\qquad W\in V_{\CC},\quad j=0,1,2,3.
\]
Consider the polynomial map
\[
\Theta:V_{\CC}\times\CC\to\CC^4
\]
given by
\[
\Theta(W,S)=\bigl(R_0(W)+S,R_1(W)+S,R_2(W)+S,R_3(W)+S\bigr).
\]

\begin{proposition}[Normal form for the discriminant]\label{prop:normal-form}
The map $\Theta$ is a polynomial isomorphism. Its inverse is
\[
Q\longmapsto (W(Q),\DeltaA(Q)).
\]
In particular,
\[
\DeltaA(\Theta(W,S))=S.
\]
\end{proposition}

\begin{proof}
Let $Q=\Theta(W,S)$. Then
\[
Q_0=\langle W,W\rangle+S.
\]
For $j=1,2,3$,
\[
R_j(W)=R_0(W)-2\langle W,p_j\rangle+\|p_j\|^2,
\]
so
\[
Q_j=Q_0-2\langle W,p_j\rangle+\|p_j\|^2.
\]
Therefore
\[
b_j(Q)=\frac{Q_0+\|p_j\|^2-Q_j}{2}=\langle W,p_j\rangle.
\]
The Gram reconstruction recovers the original $W$, and
\[
\DeltaA(Q)=Q_0-\langle W,W\rangle=S.
\]
Conversely, for arbitrary $Q$, setting $W=W(Q)$ and $S=\DeltaA(Q)$ gives $Q_j=R_j(W)+S$. Hence $\Theta$ is a polynomial isomorphism with the claimed inverse.
\end{proof}

\begin{remark}[Orthogonal model]
If $p_1,p_2,p_3$ are an orthonormal basis of $V$, then $G=I_3$ and
\[
\DeltaA(Q_0,Q_1,Q_2,Q_3)=Q_0-\frac14\sum_{j=1}^3(Q_0+1-Q_j)^2.
\]
\end{remark}

\begin{example}[The standard tetrahedral frame]
Take
\[
A=(0,1,\mathbf i,\mathbf j).
\]
Then $V=\operatorname{span}_{\RR}\{1,\mathbf i,\mathbf j\}$ and one may take $n=\mathbf k$.  Since $1,\mathbf i,\mathbf j$ are orthonormal in $\HH\simeq\RR^4$, the preceding formula gives
\[
\DeltaA(Q_0,Q_1,Q_2,Q_3)
=Q_0-\frac14\bigl((Q_0+1-Q_1)^2+(Q_0+1-Q_2)^2+(Q_0+1-Q_3)^2\bigr).
\]
Writing
\[
F=w_0+w_1\mathbf i+w_2\mathbf j+\tau\mathbf k,
\]
with $w_0,w_1,w_2,\tau\in\OO(\CC)$, the corresponding quadratic functions are
\[
\begin{aligned}
Q_0&=w_0^2+w_1^2+w_2^2+\tau^2,\\
Q_1&=(w_0-1)^2+w_1^2+w_2^2+\tau^2,\\
Q_2&=w_0^2+(w_1-1)^2+w_2^2+\tau^2,\\
Q_3&=w_0^2+w_1^2+(w_2-1)^2+\tau^2.
\end{aligned}
\]
The discriminant recovers precisely the normal square $\tau^2$.
\end{example}

\section{Polynomial and finite-order cases}
\label{sec:finite-order}

Throughout this section the four points are affinely independent and normalized as in Section~\ref{sec:discriminant}. Let $F=W+\tau n$ be the stem function and
\[
Q_j=\langle F-p_j,F-p_j\rangle.
\]
The polynomial case isolates the role of zero-free quadratic data; the finite-order result is the form needed in the toroidal argument.

\begin{lemma}[Constant $Q_j$-data]\label{lem:constant-data}
If $Q_0,Q_1,Q_2,Q_3$ are all constant entire functions, then $F=W+\tau n$ is constant.
\end{lemma}

\begin{proof}
If the $Q_j$ are constant, then all $b_j(Q)$ are constant. By the Gram reconstruction formula, $W$ is constant. Since
\[
Q_0=\langle W,W\rangle+\tau^2,
\]
it follows that $\tau^2$ is constant. Hence $\tau$ is constant: if the constant is nonzero, $\tau$ takes values in the two-point set of its square roots, and if it is zero then $\tau\equiv0$. Therefore $F$ is constant.
\end{proof}

\begin{theorem}[Polynomial case]\label{thm:polynomial-case}
Let $f(q)=\sum_{m=0}^Nq^m\alpha_m$ be a slice regular polynomial. If $f$ omits four affinely independent values, then $f$ is constant.
\end{theorem}

\begin{proof}
After translating the target, let $F$ be the stem polynomial associated with $f-a_0$. Then each
\[
Q_j(z)=\langle F(z)-p_j,F(z)-p_j\rangle
\]
is a complex polynomial. Since $f$ omits $a_j$, Lemma \ref{lem:zero-divisor} gives that $Q_j$ is zero-free. A zero-free complex polynomial is constant. Hence all $Q_j$ are constant. Lemma \ref{lem:constant-data} implies that $F$ is constant, and therefore $f$ is constant.
\end{proof}

\subsection{Finite-order zero-free functions}

For a complex entire function $h$ write
\[
\rho(h)=\limsup_{r\to\infty}\frac{\log\log M(r,h)}{\log r},
\qquad
M(r,h)=\max_{|z|=r}|h(z)|,
\]
with the convention $\rho(h)=0$ for constant functions.  If $F=\sum_{\ell=0}^3F_\ell e_\ell$ is the stem function of an entire slice regular function, we say that $F$, and hence the induced slice regular function, has finite order when
\[
\max_{0\le \ell\le3}\rho(F_\ell)<\infty .
\]
By the representation formula, the same finiteness condition is equivalent to the corresponding maximum-modulus convention on quaternionic balls, up to the usual comparison of suprema; see \cite{GentiliStoppatoStruppa2013Regular}.  The only property used below is the direct implication that finite order of the stem components gives finite order of the quadratic functions $Q_j$.

\begin{lemma}[Finite order passes to the quadratic data]\label{lem:finite-order-passes}
Let $F:\CC\to\HHc$ be a finite-order entire stem function.  Then for every $p\in\HH$,
\[
Q_p(z)=\langle F(z)-p,F(z)-p\rangle
\]
is a finite-order entire function.
\end{lemma}

\begin{proof}
Write $F=\sum_{\ell=0}^3F_\ell e_\ell$.  Let $\rho$ be larger than the orders of all components $F_\ell$.  For every $\varepsilon>0$ and all sufficiently large $r$,
\[
\max_{|z|\le r}|F_\ell(z)|\le \exp(r^{\rho+\varepsilon})
\]
for all $\ell$.  Since $Q_p$ is a finite sum of products of the $F_\ell$ and constants,
\[
\max_{|z|\le r}|Q_p(z)|\le C\exp(2r^{\rho+\varepsilon})
\]
for large $r$.  Hence $Q_p$ has finite order.
\end{proof}

\begin{lemma}[Zero-free finite-order entire functions]\label{lem:finite-order-zero-free}
Let $Q\in\OO(\CC)^*$ be zero-free and of finite order. Then there exists a complex polynomial $P\in\CC[z]$ such that $Q=e^P$. If $Q(x)>0$ for all $x\in\RR$, then $P$ may be chosen with real coefficients.
\end{lemma}

\begin{proof}
The first assertion is the zero-free case of the Hadamard factorization theorem for finite-order entire functions; see, for instance, \cite[Ch.~I, \S\S~8--9]{Boas1954EntireFunctions} and \cite[Ch.~I, \S\S~1--3]{Levin1980Zeros}.  Since $\CC$ is simply connected and $Q$ is zero-free, $Q=e^g$ for some entire function $g$.  If $Q$ has order at most $\rho$, then the Borel--Carath\'eodory inequality applied to $g-g(0)$ gives polynomial growth for $g$.  Cauchy's estimates at the origin then give $g^{(m)}(0)=0$ for all sufficiently large $m$, and hence $g$ is a polynomial.

Write $P=g$.  If $Q(x)>0$ for real $x$, then $e^{P(x)}=Q(x)>0$, so $\operatorname{Im}P(x)\in2\pi\ZZ$.  By continuity this integer is constant on $\RR$.  Subtracting the corresponding constant $2\pi\mathrm i m$ from $P$, which does not change $e^P$, we may assume $P(x)\in\RR$ for all real $x$.  A polynomial real-valued on the real axis has real coefficients.
\end{proof}

\begin{lemma}[Simple-zero obstruction]\label{lem:simple-zero}
Let $P\in\CC[z]$ be nonconstant and let $C\in\CC^*$. Then neither $e^P-C$ nor $1-Ce^P$ is an entire square.
\end{lemma}

\begin{proof}
Choose a logarithm $\log C$. The equation $e^{P(z)}=C$ is equivalent to
\[
P(z)=\log C+2\pi\mathrm i m
\]
for some $m\in\ZZ$. The set of critical values
$\CritVal(P)=P(\{P'=0\})$ is finite, while the set
$\{\log C+2\pi\mathrm i m:m\in\ZZ\}$ is infinite. Choose $m$ such that
\[
c_m=\log C+2\pi\mathrm i m
\]
is not a critical value. Since $P$ is surjective, there is $z_0$ with $P(z_0)=c_m$, and then $P'(z_0)\neq0$. Thus
\[
(e^P-C)'(z_0)=e^{P(z_0)}P'(z_0)=CP'(z_0)\neq0.
\]
So $e^P-C$ has a simple zero and cannot be an entire square. Since $1-Ce^P=-C(e^P-1/C)$, the second assertion follows.
\end{proof}

\subsection{Distance estimates}

For real $x$, real-symmetry gives $F(x)\in\HH$. Therefore
\[
Q_j(x)=\|F(x)-p_j\|^2.
\]
If $Q_j$ is zero-free, then $Q_j(x)>0$.

\begin{lemma}[Distance comparison]\label{lem:distance-comparison}
For distinct indices $i\neq k$, set
\[
d_{ik}:=\|p_i-p_k\|^2>0.
\]
Then for every real $x$,
\[
\left|Q_i(x)-Q_k(x)-d_{ik}\right|
\le
2\sqrt{d_{ik}}\sqrt{Q_k(x)}.
\]
\end{lemma}

\begin{proof}
Put $u(x)=F(x)-p_k\in\HH$ and $v_{ik}=p_i-p_k$. Then $Q_k(x)=\|u(x)\|^2$ and
\[
Q_i(x)=\|u(x)-v_{ik}\|^2=\|u(x)\|^2-2\langle u(x),v_{ik}\rangle_{\RR}+\|v_{ik}\|^2.
\]
Therefore
\[
Q_i(x)-Q_k(x)-d_{ik}=-2\langle u(x),v_{ik}\rangle_{\RR}.
\]
Cauchy--Schwarz gives the claim.
\end{proof}

\begin{lemma}[Finite-order rigidity alternatives]\label{lem:rigidity-alternatives}
Assume that the four functions $Q_j$ are zero-free, finite-order, and arise from an omission datum. Write
\[
Q_j=e^{P_j},\qquad P_j\in\RR[z].
\]
Suppose $P_k$ is nonconstant. Along one of the two real directions $x\to+\infty$ or $x\to-\infty$, one of the following alternatives occurs:
\begin{enumerate}[label=(\Alph*)]
\item $P_k(x)\to+\infty$, and then $P_i=P_k$ for every $i\neq k$.
\item $P_k(x)\to-\infty$, and then $P_i\equiv\log d_{ik}$ for every $i\neq k$.
\end{enumerate}
\end{lemma}

\begin{proof}
Since $P_k$ is a nonconstant real polynomial, along at least one real direction it tends either to $+\infty$ or to $-\infty$.

If $P_k(x)\to+\infty$, then $Q_k(x)\to+\infty$. Lemma \ref{lem:distance-comparison} gives
\[
Q_i(x)=Q_k(x)+d_{ik}+O(\sqrt{Q_k(x)}).
\]
Thus $Q_i(x)/Q_k(x)\to1$, so
\[
e^{P_i(x)-P_k(x)}\to1.
\]
Since $P_i-P_k$ is a real polynomial, this implies $P_i-P_k\to0$ along an infinite real direction, hence $P_i=P_k$.

If $P_k(x)\to-\infty$, then $Q_k(x)\to0$. Lemma \ref{lem:distance-comparison} gives $Q_i(x)\to d_{ik}>0$. Hence $P_i(x)=\log Q_i(x)\to\log d_{ik}$. A real polynomial tending to a finite limit along an infinite real direction is constant. Therefore $P_i\equiv\log d_{ik}$.
\end{proof}

\begin{lemma}[Basepoint invariance of the normal square]\label{lem:basepoint-invariance}
Let $p_0,p_1,p_2,p_3$ be affinely independent and let $V_{\mathrm{aff}}$ be the direction space of their real affine span.  Choose a unit vector $n\in V_{\mathrm{aff}}^\perp$ and write, after the normalization $p_0=0$,
\[
F=W+\tau n,\qquad W\in (V_{\mathrm{aff}})_{\CC}.
\]
If $F-p_k=W_k+\tau_k n_k$ is the decomposition obtained after using $p_k$ as basepoint, where $n_k$ is a unit normal to $V_{\mathrm{aff}}$, then $n_k=\pm n$ and $\tau_k^2=\tau^2$.  Consequently the square in the discriminant identity is independent of the chosen basepoint.
\end{lemma}

\begin{proof}
For the normalization $p_0=0$, the direction space is $V=\operatorname{span}_{\RR}\{p_1,p_2,p_3\}$.  If $p_k$ is used as basepoint, the new direction space is
\[
\operatorname{span}_{\RR}\{p_j-p_k:j\ne k\}=V,
\]
because it is the same affine-span direction.  Hence its normal line is the same line $\CC n$, so $n_k=\pm n$.  Writing $F=W+\tau n$ with $W\in V_{\CC}$ and using $p_k\in V$, one has
\[
F-p_k=(W-p_k)+\tau n.
\]
Thus the normal coordinate changes at most by sign, and the square is unchanged.
\end{proof}

\begin{corollary}[Finite-order $Q$-data]\label{cor:finite-order-Q-data}
Let $(F,(p_0,p_1,p_2,p_3))$ be a four-point omission datum in the sense of Definition~\ref{def:omission-datum}. Suppose that each
\[
Q_j(z)=\langle F(z)-p_j,F(z)-p_j\rangle,\qquad j=0,1,2,3,
\]
is of finite order. Then $F$ is constant.
\end{corollary}

\begin{proof}
For real $x$, real-symmetry gives $F(x)\in\HH$, and hence
\[
Q_j(x)=\|F(x)-p_j\|^2>0.
\]
By Lemma~\ref{lem:finite-order-zero-free},
\[
Q_j=e^{P_j},\qquad P_j\in\RR[z].
\]
If all $P_j$ are constant, then all $Q_j$ are constant, and Lemma~\ref{lem:constant-data} implies that $F$ is constant.

Assume that some $P_k$ is nonconstant. By Lemma~\ref{lem:rigidity-alternatives}, either Alternative A or Alternative B holds.  The rebasing at $p_k$ used below preserves the normal square by Lemma~\ref{lem:basepoint-invariance}.

In Alternative A, all $Q_j$ are equal to $U=e^{P_k}$. Use $p_k$ as the basepoint. Write
\[
\{i_1,i_2,i_3\}=\{0,1,2,3\}\setminus\{k\},\qquad
r_\alpha=p_{i_\alpha}-p_k\quad(1\le \alpha\le3),
\]
and let $G_k$ be the Gram matrix of $r_1,r_2,r_3$. Then all four $Q$-functions are $U$, so
\[
b_\alpha=\frac{U+\|r_\alpha\|^2-U}{2}=\frac{\|r_\alpha\|^2}{2}.
\]
With $d_k=(\|r_1\|^2,\|r_2\|^2,\|r_3\|^2)^T$, the discriminant identity gives
\[
\tau^2=U-\frac14d_k^TG_k^{-1}d_k.
\]
The constant $C_k=\frac14d_k^TG_k^{-1}d_k$ is positive, since $G_k^{-1}$ is positive definite and $d_k\neq0$. Hence
\[
\tau^2=e^{P_k}-C_k,
\]
contradicting Lemma~\ref{lem:simple-zero}.

In Alternative B, the other three functions are constants:
\[
Q_i\equiv d_{ik}=\|p_i-p_k\|^2\qquad (i\neq k).
\]
Again using $p_k$ as basepoint, put $U=Q_k=e^{P_k}$. Then
\[
b=\frac U2(1,1,1)^T.
\]
Thus
\[
\tau^2=U-\frac{U^2}{4}(1,1,1)G_k^{-1}(1,1,1)^T.
\]
Let
\[
D_k=\frac14(1,1,1)G_k^{-1}(1,1,1)^T>0.
\]
Then
\[
\tau^2=U(1-D_kU)=e^{P_k}(1-D_ke^{P_k}).
\]
Since $e^{P_k}$ has the entire square root $e^{P_k/2}$, the function
\[
S:=\tau e^{-P_k/2}
\]
is entire and satisfies
\[
S^2=1-D_ke^{P_k},
\]
again contradicting Lemma~\ref{lem:simple-zero}.

Both alternatives are impossible. Hence all $P_j$ are constant, and $F$ is constant.
\end{proof}

\begin{theorem}[Finite-order case]\label{thm:finite-order}
Let $a_0,a_1,a_2,a_3\in\HH$ be affinely independent. Let $f:\HH\to\HH$ be an entire slice regular function of finite order. If $f$ omits $a_0,a_1,a_2,a_3$, then $f$ is constant.
\end{theorem}

\begin{proof}
Translate by $-a_0$ and let $F$ be the stem function of $f-a_0$. Define $p_0=0$ and $p_j=a_j-a_0$ for $1\le j\le3$.  By Lemma~\ref{lem:zero-divisor}, the functions
\[
Q_j=\langle F-p_j,F-p_j\rangle
\]
are zero-free, so $(F,(p_j)_{j=0}^3)$ is a four-point omission datum.  By Lemma~\ref{lem:finite-order-passes}, each $Q_j$ has finite order.  Corollary~\ref{cor:finite-order-Q-data} gives that $F$ is constant, and hence $f$ is constant.
\end{proof}

\section{Toroidal reduction}
\label{sec:toroidal}

No growth condition is imposed on the original slice regular function in the rest of the proof.  Finite-order curves appear only as auxiliary exponential curves in the torus, through Corollary~\ref{cor:finite-order-Q-data}.

Let $(F,(p_0,p_1,p_2,p_3))$ be a four-point omission datum. Define
\[
Q_j(z)=\langle F(z)-p_j,F(z)-p_j\rangle,
\]
so each $Q_j\in\OO(\CC)^*$. Thus
\[
Q=(Q_0,Q_1,Q_2,Q_3):\CC\to(\CC^*)^4
\]
is an entire curve.

Define its toric Zariski closure by
\[
M_Q:=\overline{Q(\CC)}^{\Zar}\subset(\CC^*)^4,
\]
and its toric rank by
\[
\rtor(Q):=\dim M_Q.
\]

\begin{lemma}[Positivity on the real axis]\label{lem:positive-real-axis}
For every real $x$ and every $j$,
\[
Q_j(x)>0.
\]
\end{lemma}

\begin{proof}
Since $F$ is real-symmetric, $F(x)\in\HH$ for $x\in\RR$. Hence
\[
Q_j(x)=\|F(x)-p_j\|^2.
\]
The function $Q_j$ is zero-free, so it is strictly positive on the real axis.
\end{proof}

\begin{proposition}[Positive real torus parametrization]\label{prop:positive-torus}
Let $Q:\CC\to(\CC^*)^4$ come from a real-symmetric four-point omission datum.  Then $M_Q$ is a positive real translate of an algebraic subtorus: if $d=\dim M_Q$, there exist $\lambda_j\in\RR_{>0}$, $m_j\in\ZZ^d$, and a Zariski-dense real-symmetric entire curve $E:\CC\to(\CC^*)^d$ such that
\[
Q_j=\lambda_jE^{m_j},\qquad j=0,1,2,3,
\]
the exponent vectors $m_0,\ldots,m_3$ generate the lattice $\ZZ^d$, and
\[
E(x)\in(\RR_{>0})^d\qquad (x\in\RR).
\]
\end{proposition}

\begin{proof}
By the logarithmic Bloch--Ochiai theorem for semi-abelian varieties, due to Noguchi \cite[Main Theorem~(i)]{Noguchi1998HolomorphicCurves} and used here together with the formulations in \cite{NoguchiWinkelmannYamanoi2002SMT,DethloffLu2001LogarithmicJetBundles}, $M_Q$ is a translate of an algebraic subtorus of $(\CC^*)^4$. Choose $x_0\in\RR$. By Lemma \ref{lem:positive-real-axis}, $\lambda:=Q(x_0)$ belongs to $(\RR_{>0})^4$.  Since $\lambda\in M_Q$, the translate may be written canonically as
\[
M_Q=\lambda H,
\qquad
H:=\lambda^{-1}M_Q,
\]
where $H$ is an algebraic subtorus.

Choose an algebraic group isomorphism
\[
\iota:(\CC^*)^d\to H
\]
given by monomials
\[
\iota(u)=(u^{m_0},u^{m_1},u^{m_2},u^{m_3}),\qquad m_j\in\ZZ^d.
\]
Since $\iota$ is an isomorphism onto $H$, the coordinate characters $u^{m_0},\ldots,u^{m_3}$ generate the character lattice of $(\CC^*)^d$; in the chosen basis, $m_0,\ldots,m_3$ generate $\ZZ^d$.
Define
\[
E(z)=\iota^{-1}(\lambda^{-1}Q(z)).
\]
Then $Q_j=\lambda_jE^{m_j}$, and $E$ is Zariski dense because $Q$ is Zariski dense in $M_Q$. For real $x$, $\lambda^{-1}Q(x)\in H\cap(\RR_{>0})^4$.  Lemma~\ref{lem:appendix-positive-real} identifies this positive real locus with $\iota((\RR_{>0})^d)$; equivalently, the monomial inverse preserves positivity because the exponent vectors generate the full character lattice.  Hence $E(x)\in(\RR_{>0})^d$. The same lemma gives the real-symmetry of $E$.
\end{proof}

Assume $d=\rtor(Q)>0$ and write
\[
Q_j=\lambda_jE^{m_j},\qquad \lambda_j>0,\quad m_j\in\ZZ^d.
\]
Define the restricted discriminant on the torus $M_Q$ by
\[
D_M(u):=\DeltaA(\lambda_0u^{m_0},\lambda_1u^{m_1},\lambda_2u^{m_2},\lambda_3u^{m_3}).
\]
Then
\[
D_M\in\RR[u_1^{\pm1},\dots,u_d^{\pm1}],
\]
and the square-discriminant equation becomes
\[
D_M(E(z))=\tau(z)^2.
\]

\begin{lemma}[Rank zero]\label{lem:rank-zero}
If $\rtor(Q)=0$, then $F$ is constant.
\end{lemma}

\begin{proof}
If $\rtor(Q)=0$, then all $Q_j$ are constant. Lemma \ref{lem:constant-data} implies that $F$ is constant.
\end{proof}

\begin{lemma}[Finite-order dense exponential curves]\label{lem:dense-exponential}
Let $d\ge1$, and choose real numbers $\alpha_1,\dots,\alpha_d$ linearly independent over $\QQ$. Define
\[
\gamma_\alpha(z)=(e^{\alpha_1z},\dots,e^{\alpha_dz}).
\]
Then $\gamma_\alpha$ is Zariski dense in $(\CC^*)^d$, and each coordinate is real-symmetric and positive on the real axis. Moreover, if $m\in\ZZ^d$ is nonzero, then the character $\gamma_\alpha^m$ is nonconstant.
\end{lemma}

\begin{proof}
If a Laurent polynomial $P(u)=\sum_{\nu\in S}c_\nu u^\nu$ vanishes on $\gamma_\alpha$, then
\[
0=P(\gamma_\alpha(z))=\sum_{\nu\in S}c_\nu e^{\langle\nu,\alpha\rangle z}.
\]
The numbers $\langle\nu,\alpha\rangle$ are distinct for distinct $\nu$, by $\QQ$-linear independence. The corresponding exponentials are linearly independent, so all $c_\nu$ vanish. The final assertion follows because $\langle m,\alpha\rangle\neq0$ for every nonzero $m\in\ZZ^d$.
\end{proof}

\begin{proposition}\label{prop:DM-not-zero}
Assume a nonconstant four-point omission datum exists with positive toric rank $d=\rtor(Q)>0$, and let $D_M$ be the restricted discriminant on its toric Zariski closure. Then
\[
D_M\not\equiv0.
\]
\end{proposition}

\begin{proof}
Suppose $D_M\equiv0$. Choose a finite-order Zariski-dense exponential curve $\gamma_\alpha$ as in Lemma \ref{lem:dense-exponential}. Define
\[
\widetilde Q_j(z)=\lambda_j\gamma_\alpha(z)^{m_j}=\lambda_j\exp(\langle m_j,\alpha\rangle z).
\]
Then each $\widetilde Q_j$ is real-symmetric, zero-free, and finite-order, and
\[
\DeltaA(\widetilde Q_0,\widetilde Q_1,\widetilde Q_2,\widetilde Q_3)=0.
\]
Taking $\widetilde\tau\equiv0$, Proposition~\ref{prop:reconstruction} produces an entire stem function with zero-free finite-order quadratic data. Since $d>0$, not all exponent vectors $m_j$ are zero; otherwise the torus parametrization would be constant. Choose an index $j$ with $m_j\ne0$. For the $\QQ$-linearly independent vector $\alpha$ of Lemma~\ref{lem:dense-exponential}, one has $\langle m_j,\alpha\rangle\ne0$, so $\widetilde Q_j$ is nonconstant.  If the reconstructed stem function were constant, all four functions $\langle F-p_j,F-p_j\rangle$ would be constant, contrary to the prescribed nonconstant $\widetilde Q_j$.  Thus the reconstructed stem function is nonconstant, contradicting Corollary~\ref{cor:finite-order-Q-data}.
\end{proof}

\begin{lemma}[Real descent of Laurent squares]\label{lem:real-descent}
Let $D\in\RR[u_1^{\pm1},\dots,u_d^{\pm1}]$ be nonzero. Suppose that over the complex Laurent ring
\[
D=c\,u^\ell S(u)^2,
\]
where $c\in\CC^*$, $\ell\in\ZZ^d$, and $S\in\CC[u_1^{\pm1},\dots,u_d^{\pm1}]$. Then there exist $a\in\RR^*$, $\ell_0\in\ZZ^d$, and $S_0\in\RR[u_1^{\pm1},\dots,u_d^{\pm1}]$ such that
\[
D=a\,u^{\ell_0} S_0(u)^2.
\]
\end{lemma}

\begin{proof}
Put
\[
R=\RR[u_1^{\pm1},\dots,u_d^{\pm1}],\qquad
R_{\CC}=R\otimes_{\RR}\CC.
\]
Both rings are UFDs; the units of $R$ are $a u^m$ with $a\in\RR^*$ and $m\in\ZZ^d$.  Write the factorization of $D$ in $R$ as
\[
D=a u^m\prod_{\beta} f_\beta^{e_\beta},
\]
where $a\in\RR^*$ and the $f_\beta$ are pairwise nonassociated real irreducible Laurent polynomials which are not units.  Fix one factor $f=f_\beta$.

Over $R_{\CC}$ there are two possibilities.  Either $f$ remains irreducible, in which case the multiplicity of this complex irreducible factor in $D$ is $e_\beta$.  Or $f$ is reducible over $\CC$.  Galois descent in the UFD $R_{\CC}$ gives, after multiplication by a Laurent unit,
\[
f=b\,h\,\overline h
\]
with $h$ complex irreducible.  The factors $h$ and $\overline h$ are nonassociated.  Indeed, if $\overline h=c\,u^k h$ with $c\in\CC^*$ and $k\in\ZZ^d$, then conjugating again gives $h=|c|^2u^{2k}h$, whence $k=0$ and $|c|=1$.  Writing $c=e^{2\mathrm i\theta}$, the Laurent polynomial $g=e^{\mathrm i\theta}h$ satisfies $\overline g=g$, so $g\in R$, and $g$ is a proper nonunit factor of $f$ in $R$, contradicting the irreducibility of $f$ over $R$.  Since $h$ cannot divide the complex factorization of a different nonassociated real irreducible factor, the multiplicities of $h$ and $\overline h$ in $D$ are both $e_\beta$.  As $D$ is a square up to a Laurent unit in $R_{\CC}$, every complex irreducible factor appearing in $D$ has even multiplicity.  Hence $e_\beta$ is even in either case.

Thus all non-unit real irreducible factors of $D$ have even exponent.  Therefore
\[
D=a u^m S_0^2
\]
for some $S_0\in R$.  Taking $\ell_0=m$ gives the claimed form.
\end{proof}

\begin{proposition}\label{prop:DM-not-square}
Assume a nonconstant four-point omission datum exists with positive toric rank $d=\rtor(Q)>0$. Then the restricted discriminant $D_M$ cannot be written as
\[
D_M=c\,u^\ell S(u)^2
\]
with $c\in\CC^*$, $\ell\in\ZZ^d$, and $S\in\CC[u_1^{\pm1},\dots,u_d^{\pm1}]$.
\end{proposition}

\begin{proof}
A nonzero constant Laurent polynomial is a Laurent unit times a square, so the assertion also excludes the constant nonzero case.  Assume that $D_M$ is such a Laurent unit times a square. Since $D_M$ has real coefficients, Lemma \ref{lem:real-descent} gives
\[
D_M=a\,u^{\ell_0} S_0(u)^2
\]
with $a\in\RR^*$, $\ell_0\in\ZZ^d$, and $S_0$ real Laurent. Evaluate on the actual curve $E$. For real $x$, $E(x)\in(\RR_{>0})^d$, so $E(x)^{\ell_0}>0$ and, because $S_0$ has real coefficients, $S_0(E(x))^2\ge0$.  The normal coordinate is real on the real axis, and hence
\[
D_M(E(x))=\tau(x)^2\ge0.
\]
If $a<0$, the identity
\[
\tau(x)^2=aE(x)^{\ell_0}S_0(E(x))^2
\]
forces both sides to vanish for every real $x$.  Thus $S_0(E(x))=0$ for all real $x$, hence $S_0(E)\equiv0$. By Zariski density, $S_0\equiv0$, contradicting Proposition \ref{prop:DM-not-zero}. Therefore $a>0$.

Choose a finite-order Zariski-dense exponential curve $\gamma_\alpha$ as in Lemma \ref{lem:dense-exponential}. Define
\[
\widetilde Q_j(z)=\lambda_j\exp(\langle m_j,\alpha\rangle z)
\]
and
\[
\widetilde\tau(z)=\sqrt a\,\exp\left(\frac12\langle\ell_0,\alpha\rangle z\right)S_0(e^{\alpha_1z},\dots,e^{\alpha_dz}).
\]
Then
\[
\DeltaA(\widetilde Q_0,\widetilde Q_1,\widetilde Q_2,\widetilde Q_3)=\widetilde\tau^2.
\]
The functions $\widetilde Q_j$ are real-symmetric because $\lambda_j>0$ and $\alpha\in\RR^d$; the function $\widetilde\tau$ is real-symmetric because $a>0$ and $S_0$ has real Laurent coefficients.  Since the exponent vectors $m_0,\ldots,m_3$ generate $\ZZ^d$ and $d>0$, at least one $m_j$ is nonzero. By Lemma~\ref{lem:dense-exponential}, the corresponding $\widetilde Q_j$ is nonconstant.  Proposition~\ref{prop:reconstruction} gives a stem function whose four quadratic functions are precisely the $\widetilde Q_j$; since one of them is nonconstant, the stem function itself is nonconstant.  Its four quadratic functions are zero-free and of finite order, contradicting Corollary~\ref{cor:finite-order-Q-data}.
\end{proof}

\section{Laurent branch divisors and the NWY ramification obstruction}
\label{sec:laurent-ramification}

The form of the theorem of Noguchi--Winkelmann--Yamanoi needed here is the following divisor case.  Here $T_g(r;L)$ denotes the Nevanlinna characteristic with respect to a line bundle $L$, and $N_1$ denotes the counting function truncated at level one.

\begin{notation}\label{not:epsilon-exception}
For fixed $\varepsilon>0$, the notation
\[
A(r)\le B(r)\quad \|_\varepsilon
\]
means that the estimate holds for all $r$ outside an exceptional subset of $(0,\infty)$ of finite Lebesgue measure, depending on $\varepsilon$; this is the convention used in the theorem of Noguchi--Winkelmann--Yamanoi.
\end{notation}

\begin{theorem}[Noguchi--Winkelmann--Yamanoi, divisor case]\label{thm:nwy-input}
Let $A$ be a semi-abelian variety, let $g:\CC\to A$ be a Zariski-dense entire curve, and let $\mathcal D\subset A$ be an effective reduced divisor. Then there is a smooth equivariant compactification
\[
\overline A=\overline A(A,\mathcal D)
\]
of $A$ such that, for the closure $\overline{\mathcal D}$ and every $\varepsilon>0$,
\begin{equation}\label{eq:nwy-inequality}
T_g(r;L(\overline{\mathcal D}))\le N_1(r;g^*\mathcal D)+\varepsilon T_g(r;L(\overline{\mathcal D}))\quad\|_\varepsilon.
\end{equation}
\end{theorem}

This is the divisor case of \cite[Main Theorem~(iii)]{NoguchiWinkelmannYamanoi2008SMTII}.  In the applications $A$ is an algebraic torus; smooth equivariant compactifications are then smooth toric compactifications, and any two are dominated by a common smooth projective fan refinement.

\begin{remark}[Equivariant refinements]\label{rem:equivariant-refinement}
If $\mu:X'\to\overline A$ is a smooth equivariant refinement and $g':\CC\to X'$ is the lift of $g$ through the common open torus, then the open divisor pullback is unchanged:
\[
(g')^*(\mu^*\overline{\mathcal D}|_A)=g^*\mathcal D.
\]
Consequently the truncated and untruncated counting functions for the open divisor are the same before and after refinement; boundary components do not meet the open-torus image of the curve and contribute no counting term.  Functoriality of characteristics gives
\[
T_{g'}(r;\mu^*L(\overline{\mathcal D}))=T_g(r;L(\overline{\mathcal D}))+O(1).
\]
Finally, if $\overline{\mathcal D}'$ is the strict closure of the open divisor on $X'$, then
\[
\mu^*\overline{\mathcal D}=\overline{\mathcal D}'+R
\]
as Cartier divisors for some effective divisor $R$ supported in the boundary and the exceptional locus.  Hence bigness of $L(\overline{\mathcal D}')$ implies bigness of $\mu^*L(\overline{\mathcal D})$.
\end{remark}

Let
\[
D\in\CC[u_1^{\pm1},\dots,u_d^{\pm1}]
\]
be a nonzero Laurent polynomial. Since this Laurent ring is a unique factorization domain, write
\[
D(u)=c\,u^\ell\prod_{\alpha=1}^NR_\alpha(u)^{e_\alpha},
\]
where $c\in\CC^*$, $\ell\in\ZZ^d$, and the $R_\alpha$ are pairwise nonassociated irreducible Laurent polynomials. Define the squarefree branch part of $D$ by
\[
B_D(u):=\prod_{e_\alpha\text{ odd}}R_\alpha(u).
\]
Thus
\[
D=c\,u^\ell S(u)^2B_D(u)
\]
for some Laurent polynomial $S$. The branch part $B_D$ is well-defined only up to Laurent units. It is a Laurent unit if and only if $D$ is a Laurent unit times a square.

\begin{lemma}[Even branch pullback]\label{lem:even-branch}
Let $E:\CC\to(\CC^*)^d$ be an entire curve. Assume
\[
D(E)\not\equiv0
\]
and
\[
D(E)=\tau^2
\]
for some entire function $\tau$. Then every zero of $B_D(E)$ has even multiplicity.
\end{lemma}

\begin{proof}
Pull back the factorization
\[
D=c\,u^\ell S^2B_D.
\]
Let $z_0$ be a zero of $B_D(E)$.  Since $D(E)\not\equiv0$, neither $S(E)$ nor $B_D(E)$ is identically zero, and all local vanishing orders at $z_0$ are finite.  The factor $E^\ell$ is zero-free, hence
\[
\ord_{z_0}D(E)=2\ord_{z_0}S(E)+\ord_{z_0}B_D(E).
\]
The left side is even because $D(E)=\tau^2$, and the first term on the right is even.  Therefore $\ord_{z_0}B_D(E)$ is even.
\end{proof}

\subsection{Stabilizers of Laurent divisors}

Let $A=(\CC^*)^d$. For an effective divisor $\mathcal D\subset A$, define its connected stabilizer by
\[
\St(\mathcal D)=\{a\in A:a\cdot\mathcal D=\mathcal D\}^0.
\]
The full stabilizer is a closed algebraic subgroup of $A$, so its identity component is well-defined.  The following quotient step is the torus form of the stabilizer reduction used in \cite[Propositions~3.10--3.11]{NoguchiWinkelmannYamanoi2008SMTII}.

\begin{lemma}[Stabilizer quotient]\label{lem:stabilizer-quotient}
Let $B\in\CC[u_1^{\pm1},\dots,u_d^{\pm1}]$ be a reduced Laurent polynomial which is not a Laurent unit, and let $\mathcal D_B=\{B=0\}\subset(\CC^*)^d$. Let $K=\St(\mathcal D_B)$. Then $K\ne(\CC^*)^d$. In particular, there exist a quotient torus
\[
\pi:(\CC^*)^d\to(\CC^*)^r,\qquad r\ge1,
\]
and a reduced Laurent divisor $\widetilde{\mathcal D}\subset(\CC^*)^r$ with trivial connected stabilizer such that
\[
\mathcal D_B=\pi^{-1}(\widetilde{\mathcal D})
\]
as reduced divisors. Moreover, if $E:\CC\to(\CC^*)^d$ is Zariski dense, then
$G=\pi\circ E$ is Zariski dense in $(\CC^*)^r$.
\end{lemma}

\begin{proof}
Since $B$ is not a Laurent unit, $\mathcal D_B$ is a nonempty proper divisor in the torus. If $K=(\CC^*)^d$, then $\mathcal D_B$ would be invariant under all translations of the torus; the only such reduced closed subsets are $\varnothing$ and the whole torus. This is impossible for a nonempty proper divisor. Hence $K\ne(\CC^*)^d$, and the quotient has dimension $r=d-\dim K\ge1$.

Since $K$ is connected and $\mathcal D_B$ has only finitely many irreducible components, $K$ fixes each component. Thus $\mathcal D_B$ is saturated with respect to the quotient map
\[
\pi:(\CC^*)^d\to(\CC^*)^d/K.
\]
It therefore descends to an effective divisor on the quotient torus. Taking the reduced structure downstairs gives a reduced divisor $\widetilde{\mathcal D}$ satisfying
\[
\mathcal D_B=\pi^{-1}(\widetilde{\mathcal D})
\]
as reduced divisors; the equality is reduced because $\pi$ is a smooth morphism of tori.

If $\widetilde{\mathcal D}$ had a positive-dimensional connected stabilizer $L$ in the quotient, then $\pi^{-1}(L)^0$ would strictly contain $K$ and would stabilize $\mathcal D_B$. This contradicts the maximality of $K=\St(\mathcal D_B)$. Hence $\widetilde{\mathcal D}$ has trivial connected stabilizer.

Finally, Zariski density is preserved under dominant quotient maps. If $\pi\circ E$ were algebraically degenerate, then $E$ would lie in the inverse image of a proper algebraic subset of the quotient.
\end{proof}

The corresponding local multiplicity comparison under the quotient map is the torus form of the same reduction in \cite[Propositions~3.10--3.11]{NoguchiWinkelmannYamanoi2008SMTII}.

\begin{lemma}[Multiplicity under the stabilizer quotient]\label{lem:multiplicity-quotient}
Let $\pi:A\to A'$ be a quotient of algebraic tori, let $\widetilde{\mathcal D}\subset A'$ be a reduced divisor, and suppose that
\[
\mathcal D=\pi^{-1}(\widetilde{\mathcal D})
\]
as reduced divisors.  If $E:\CC\to A$ is an entire curve, $G=\pi\circ E$, and every point of $E^*\mathcal D$ has multiplicity at least $2$, then every point of $G^*\widetilde{\mathcal D}$ has multiplicity at least $2$.
\end{lemma}

\begin{proof}
Let $z_0\in\CC$ with $G(z_0)\in\widetilde{\mathcal D}$.  Choose a local defining equation $h$ for $\widetilde{\mathcal D}$ near $G(z_0)$.  Since $\pi$ is a smooth morphism of tori and $\mathcal D=\pi^{-1}(\widetilde{\mathcal D})$ as reduced divisors, $h\circ\pi$ is a local defining equation for $\mathcal D$ near $E(z_0)$, up to multiplication by a unit.  Hence
\[
\operatorname{ord}_{z_0}G^*\widetilde{\mathcal D}
=
\operatorname{ord}_{z_0}(h\circ G)
=
\operatorname{ord}_{z_0}(h\circ\pi\circ E)
=
\operatorname{ord}_{z_0}E^*\mathcal D.
\]
The last order is at least $2$ by assumption.
\end{proof}

\subsection{The toric ramification obstruction}

The bigness statement used below is the toric specialization of \cite[Proposition~3.9(ii)]{NoguchiWinkelmannYamanoi2008SMTII}; the proof here uses the Newton polytope.

\begin{lemma}[Toric bigness for trivial stabilizer]\label{lem:toric-bigness}
Let $A=(\CC^*)^r$ and let $\mathcal D=\{B=0\}\subset A$ be a reduced effective divisor, where
\[
B=\sum_{m\in S}c_m u^m\in\CC[u_1^{\pm1},\dots,u_r^{\pm1}]
\]
is a nonzero Laurent polynomial.  If $\St(\mathcal D)=\{1\}$, then the Newton polytope $\Newt(B)\subset\RR^r$ is full-dimensional.  Moreover, if $X_\Sigma$ is any smooth projective toric compactification of $A$ and $\overline{\mathcal D}$ is the strict closure of $\mathcal D$ in $X_\Sigma$, then $L(\overline{\mathcal D})$ is big.  In particular this holds on every smooth projective refinement of the normal fan of $\Newt(B)$.
\end{lemma}

\begin{proof}
If $\Newt(B)$ were contained in a proper affine hyperplane, then there would be a nonzero integral cocharacter $v\in\ZZ^r$ and an exponent $m_0\in S$ such that
\[
\langle m-m_0,v\rangle=0\qquad(m\in S).
\]
For the corresponding one-parameter subgroup $\lambda_v:\CC^*\to A$, this gives
\[
B(\lambda_v(t)u)=t^{\langle m_0,v\rangle}B(u).
\]
Hence the divisor $\{B=0\}$ is invariant under the positive-dimensional subtorus $\lambda_v(\CC^*)$, contradicting the triviality of the connected stabilizer. Thus $\Newt(B)$ is full-dimensional.

Let $X_\Sigma$ be a smooth projective toric compactification with rays $\rho$ and primitive ray generators $v_\rho$.  As a rational function on $X_\Sigma$, $B$ has boundary order
\[
\min_{m\in S}\langle m,v_\rho\rangle
\]
along the boundary divisor $D_\rho$.  Hence the strict closure of $\mathcal D$ is linearly equivalent to the torus-invariant divisor
\[
D_B=-\sum_\rho \min_{m\in S}\langle m,v_\rho\rangle D_\rho .
\]
The associated polytope contains $\Newt(B)$, because every $m\in\Newt(B)$ satisfies
\[
\langle m,v_\rho\rangle\ge \min_{m'\in S}\langle m',v_\rho\rangle
\]
for all $\rho$.  Since $\Newt(B)$ is full-dimensional, the polytope of $D_B$ is full-dimensional.  The toric divisor--polytope correspondence gives
\[
h^0(X_\Sigma,\mathcal O(kD_B))
=\#(kP_{D_B}\cap\ZZ^r)
\]
for $k\gg0$; in particular this grows at least like the number of lattice points in $k\Newt(B)$, hence like $c k^r$ for some $c>0$.  Thus $D_B$, and hence $L(\overline{\mathcal D})$, is big.  On a refinement of the normal fan of $\Newt(B)$, this is the usual toric line bundle attached to $\Newt(B)$.  This is the standard correspondence between torus-invariant divisors, support functions and polytopes; see \cite[\S~3.4]{Fulton1993Toric} and \cite[Proposition~4.3.3 and Ch.~9]{CoxLittleSchenck2011Toric}.
\end{proof}

The counting function $N_1$ is truncated at level one, whereas $N$ counts multiplicities.  The exceptional-set notation is that of Notation~\ref{not:epsilon-exception}.  All characteristic functions below are computed with respect to the chosen equivariant compactification and the indicated line bundle.

The ramification obstruction also uses the bounded-characteristic criterion for big line bundles, stated as Lemma~\ref{lem:appendix-big-characteristic}.

\begin{theorem}[Toric ramification obstruction]\label{thm:toric-ramification}
Let $A=(\CC^*)^r$ and let $\mathcal D\subset A$ be an effective reduced divisor with trivial connected stabilizer:
\[
\St(\mathcal D)=\{1\}.
\]
If $G:\CC\to A$ is Zariski dense, then $G^*\mathcal D$ contains a point of multiplicity one.  Equivalently, $G$ meets $\mathcal D$ and does so somewhere with multiplicity one.
\end{theorem}

\begin{proof}
Assume on the contrary that $G^*\mathcal D$ is either empty or has all multiplicities at least $2$.  Put the NWY compactification and the toric compactification used for bigness on a common model, using the convention of Remark~\ref{rem:equivariant-refinement}.  Apply Theorem~\ref{thm:nwy-input} and let $\overline A_N$ be the smooth equivariant compactification it supplies, with $\overline{\mathcal D}_N$ the closure of $\mathcal D$.  Choose a reduced Laurent equation $B$ for $\mathcal D$.  Let $X$ be a smooth projective toric compactification which dominates $\overline A_N$ and also refines the normal fan of $\Newt(B)$; this is obtained from a common smooth projective fan refinement.  Write
\[
\mu:X\to\overline A_N
\]
for the induced toric morphism and let $G':\CC\to X$ be the lift of $G$ through the open torus.

Set
\[
L':=\mu^*L(\overline{\mathcal D}_N).
\]
Pulling back the NWY inequality \eqref{eq:nwy-inequality} to $X$ gives the same open counting functions.  Indeed, $G'(\CC)\subset A$, and the restriction of $\mu^*\overline{\mathcal D}_N$ to $A$ is exactly $\mathcal D$; hence boundary and exceptional components of $X\setminus A$ contribute no zeros to the open pullback.  Therefore
\[
N_1(r;G'^*(\mu^*\overline{\mathcal D}_N)|_A)=N_1(r;G^*\mathcal D),\qquad
N(r;G'^*(\mu^*\overline{\mathcal D}_N)|_A)=N(r;G^*\mathcal D),
\]
and functoriality of characteristic functions gives
\[
T_{G'}(r;L')=T_G(r;L(\overline{\mathcal D}_N))+O(1).
\]
Consequently, for every $\varepsilon>0$,
\[
T_{G'}(r;L')
\le
N_1(r;G^*\mathcal D)+\varepsilon T_{G'}(r;L')+O(1)\quad\|_\varepsilon.
\]

The same line bundle $L'$ is big.  Let $\overline{\mathcal D}_X$ be the strict closure of $\mathcal D$ in $X$.  Since $\mu$ is an isomorphism over the open torus,
\[
\mu^*\overline{\mathcal D}_N=\overline{\mathcal D}_X+R
\]
as Cartier divisors, where $R$ is effective and supported on boundary or exceptional divisors.  By Lemma~\ref{lem:toric-bigness}, $L(\overline{\mathcal D}_X)$ is big: the connected stabilizer of $\mathcal D$ is trivial, so $\Newt(B)$ is full-dimensional.  Adding the effective divisor $R$ preserves bigness.  Hence $L'$ is a big line bundle on $X$.

Now use the ramification hypothesis.  Since all multiplicities of $G^*\mathcal D$ are at least $2$,
\[
N_1(r;G^*\mathcal D)\le\frac12N(r;G^*\mathcal D).
\]
The First Main Theorem for the divisor corresponding to $L'$ gives
\[
N(r;G^*\mathcal D)\le T_{G'}(r;L')+O(1).
\]
Substitution in the pulled-back NWY inequality yields
\[
T_{G'}(r;L')
\le
\left(\frac12+\varepsilon\right)T_{G'}(r;L')+O(1)\quad\|_\varepsilon.
\]
Choose $0<\varepsilon<1/2$. Then
\[
T_{G'}(r;L')=O(1)\quad\|_\varepsilon.
\]
The curve $G'$ is Zariski dense in $X$, because its image in the open torus is the Zariski-dense curve $G$.  This contradicts Lemma~\ref{lem:appendix-big-characteristic}, which says that a Zariski-dense curve cannot have bounded characteristic with respect to a big line bundle outside finite-measure exceptional sets.  The contradiction proves the theorem.
\end{proof}

\begin{lemma}[Zariski density prevents identically vanishing branch pullback]\label{lem:branch-not-identically-zero}
Let $B\in\CC[u_1^{\pm1},\dots,u_d^{\pm1}]$ be a nonzero Laurent polynomial and let
\[
E:\CC\to(\CC^*)^d
\]
be Zariski dense. Then $B(E)\not\equiv0$.
\end{lemma}

\begin{proof}
If $B(E)\equiv0$, then $E(\CC)$ is contained in the proper algebraic hypersurface $\{B=0\}\subset(\CC^*)^d$, contradicting Zariski density.
\end{proof}

\begin{theorem}[Laurent even-pullback obstruction]\label{thm:laurent-even-pullback}
Let
\[
D\in\CC[u_1^{\pm1},\dots,u_d^{\pm1}]
\]
be a nonzero Laurent polynomial. Assume that $D$ is not a Laurent unit times a square. Then there do not exist a Zariski-dense entire curve $E:\CC\to(\CC^*)^d$ and an entire function $\tau$ such that
\[
D(E)=\tau^2.
\]
\end{theorem}

\begin{proof}
Let $B_D$ be the squarefree branch part of $D$. Since $D$ is not a Laurent unit times a square, $B_D$ is not a Laurent unit. Assume $D(E)=\tau^2$ for a Zariski-dense $E$. Since both $D$ and $B_D$ are nonzero Laurent polynomials and $E$ is Zariski dense, Lemma~\ref{lem:branch-not-identically-zero} gives
\[
D(E)\not\equiv0,
\qquad
B_D(E)\not\equiv0.
\]
Lemma \ref{lem:even-branch} then implies that every zero of $B_D(E)$ has even multiplicity.

Let $\mathcal D_{B_D}=\{B_D=0\}$. Quotient by the connected stabilizer as in Lemma \ref{lem:stabilizer-quotient}. We get a quotient torus $\pi:(\CC^*)^d\to(\CC^*)^r$ and a reduced divisor $\widetilde{\mathcal D}$ with trivial connected stabilizer such that $\mathcal D_{B_D}=\pi^{-1}(\widetilde{\mathcal D})$. Put $G=\pi\circ E$. Then $G$ is Zariski dense.  Lemma~\ref{lem:multiplicity-quotient} carries the even multiplicity of $E^*\mathcal D_{B_D}$ to $G^*\widetilde{\mathcal D}$, so every point of $G^*\widetilde{\mathcal D}$ has multiplicity at least $2$. This contradicts Theorem \ref{thm:toric-ramification}.
\end{proof}

\begin{remark}
The Zariski-density hypothesis in Theorem~\ref{thm:laurent-even-pullback} cannot be dropped: a curve contained in a proper subvariety of the torus may meet a locus on which the branch divisor becomes a square or disappears.  In the four-point problem, density is part of the toric reduction; after passing to torus coordinates, $E$ is dense in the torus on which $D_M$ is defined.
\end{remark}

\begin{proposition}[Toroidal square-discriminant obstruction]\label{prop:toroidal-obstruction}
There is no nonconstant four-point omission datum whose $Q$-curve has positive toric rank.
\end{proposition}

\begin{proof}
Assume $\rtor(Q)>0$. By Proposition \ref{prop:positive-torus}, write $Q_j=\lambda_jE^{m_j}$ with $E$ Zariski dense. Define $D_M$ as in Section~\ref{sec:toroidal}. Then $D_M(E)=\tau^2$. Proposition \ref{prop:DM-not-zero} gives $D_M\not\equiv0$, and Proposition \ref{prop:DM-not-square} says that $D_M$ is not a Laurent unit times a square. This contradicts Theorem \ref{thm:laurent-even-pullback}.
\end{proof}

\section{Proof of the four-point Picard theorem}
\label{sec:proof}

\begin{proof}[Proof of Theorem \ref{thm:main}]
Assume for contradiction that a nonconstant entire slice regular function $f:\HH\to\HH$ omits four affinely independent values $a_0,a_1,a_2,a_3$.

Translate by $-a_0$ and let $F$ be the real-symmetric entire stem function inducing $f-a_0$. Define
\[
Q_j(z)=\langle F(z)-p_j,F(z)-p_j\rangle,
\qquad j=0,1,2,3.
\]
By Lemma \ref{lem:zero-divisor}, each $Q_j$ is zero-free, and hence $Q=(Q_0,Q_1,Q_2,Q_3)$ defines a curve in $(\CC^*)^4$. By Theorem \ref{thm:discriminant-reduction},
\[
\DeltaA(Q_0,Q_1,Q_2,Q_3)=\tau^2.
\]
Let $M_Q=\overline{Q(\CC)}^{\Zar}$ and $\rtor(Q)=\dim M_Q$.

If $\rtor(Q)=0$, then Lemma~\ref{lem:rank-zero} gives that $F$ is constant, so $f$ is constant, contradiction.

If $\rtor(Q)>0$, then Proposition \ref{prop:toroidal-obstruction} gives a contradiction.

Both alternatives lead to contradictions for a nonconstant $f$. Hence every entire slice regular function omitting four affinely independent quaternionic values is constant.
\end{proof}

\begin{proof}[Proof of Corollary \ref{cor:exact-four-value}]
If the four points are affinely dependent, they lie in a real affine subspace of dimension at most $2$, hence in a real affine $2$-plane. The Bisi--Winkelmann construction gives nonconstant entire slice regular functions omitting any prescribed real affine $2$-plane \cite{BisiWinkelmann2020QuaternionicPicard}. Thus the four points can be omitted by a nonconstant entire slice regular function.

If the four points are affinely independent, Theorem \ref{thm:main} says that any entire slice regular function omitting them is constant.
\end{proof}

\appendix
\section{Toric and Nevanlinna lemmas}
\label{app:technical}

This appendix contains the toric and Nevanlinna facts used in Sections~\ref{sec:toroidal} and~\ref{sec:laurent-ramification}.

\begin{lemma}[Positive real parametrization of subtori]\label{lem:appendix-positive-real}
Let $H\subset (\CC^*)^N$ be an algebraic subtorus of dimension $d$. There is a monomial isomorphism
\[
\iota:(\CC^*)^d\longrightarrow H,
\qquad
\iota(u)=(u^{m_1},\dots,u^{m_N}),
\]
with $m_j\in\ZZ^d$, and this isomorphism identifies the positive real locus $(\RR_{>0})^d$ with
\[
H\cap(\RR_{>0})^N.
\]
Consequently, if $M=\lambda H$ with $\lambda\in(\RR_{>0})^N$ and $Q:\CC\to M$ is an entire curve such that $Q(x)\in(\RR_{>0})^N$ for every $x\in\RR$, then the lifted curve
\[
E:=\iota^{-1}(\lambda^{-1}Q)
\]
satisfies $E(x)\in(\RR_{>0})^d$ for every $x\in\RR$ and is real-symmetric.
\end{lemma}

\begin{proof}
The character lattice of $(\CC^*)^N$ is $\ZZ^N$. Choosing a basis of the character lattice of $H$ identifies $H$ with $(\CC^*)^d$, and the restrictions of the ambient coordinate characters become monomials $u^{m_j}$. This gives the displayed isomorphism.  Since $\iota$ is an isomorphism, not merely an isogeny, the exponent vectors $m_1,\ldots,m_N$ generate the full lattice $\ZZ^d$.

The inclusion $\iota((\RR_{>0})^d)\subset H\cap(\RR_{>0})^N$ is immediate.  Conversely, if $u\in(\CC^*)^d$ and $\iota(u)\in(\RR_{>0})^N$, then $u^{m_j}>0$ for all $j$.  Since the $m_j$ generate $\ZZ^d$, each coordinate character $u_a$ is an integral product of the characters $u^{m_j}$, and hence $u_a>0$.  Thus $u\in(\RR_{>0})^d$, proving that $\iota$ identifies the positive real loci.

Thus $E(x)\in(\RR_{>0})^d$ for real $x$. For each component $E_a$, the holomorphic functions $E_a(z)$ and $\overline{E_a(\overline z)}$ agree on the real axis; by the identity theorem they agree on all of $\CC$. Hence $E$ is real-symmetric.
\end{proof}

\begin{lemma}[Big characteristic detects Zariski density]\label{lem:appendix-big-characteristic}
Let $X$ be an irreducible normal projective variety, let $L$ be a big line bundle on $X$, and let
$g:\CC\to X$ be a holomorphic curve.  If
\begin{equation*}
T_g(r;L)=O(1)
\end{equation*}
outside an exceptional set of finite Lebesgue measure, then $g$ is algebraically degenerate.  Equivalently, for a Zariski-dense curve, the characteristic with respect to a big line bundle cannot be bounded outside such exceptional sets.
\end{lemma}

\begin{proof}
Assume that the image of $g$ is Zariski dense.  If $\dim X=0$, then $g$ is constant, so we may assume $\dim X\ge1$.  Since $L$ is big, Kodaira's lemma gives an integer $m>0$, an ample line bundle $H$, and an effective divisor $E_0$ on $X$ such that
\begin{equation*}
mL\cong H\otimes\OO_X(E_0).
\end{equation*}
The curve is Zariski dense, hence it is not contained in $\operatorname{Supp}E_0$.  The First Main Theorem for the effective divisor $E_0$ gives
\begin{equation*}
T_g(r;\OO_X(E_0))
=
N(r;g^*E_0)+m_g(r;E_0)+O(1)
\ge -C_1,
\end{equation*}
for a constant $C_1$, because the counting function is nonnegative and the proximity function of an effective divisor is bounded from below.  Therefore
\begin{equation*}
T_g(r;H)
=
mT_g(r;L)-T_g(r;\OO_X(E_0))+O(1)
\le mT_g(r;L)+C_2.
\end{equation*}
By hypothesis, there are a constant $C_0$ and an exceptional set $S\subset(0,\infty)$ of finite Lebesgue measure such that $T_g(r;L)\le C_0$ for $r\notin S$.  Fix a smooth metric on $H$ with positive curvature form.  Then $r\mapsto T_g(r;H)$ is nondecreasing up to a bounded term.  For arbitrary $r>0$, choose
\begin{equation*}
r'\in [r,r+|S|+1]\setminus S,
\end{equation*}
which is possible since the interval length is larger than the measure of $S$.  Then
\begin{equation*}
T_g(r;H)\le T_g(r';H)+O(1)\le mC_0+C_3
\end{equation*}
for all $r>0$.  Thus $T_g(r;H)$ is bounded.

Choose $k>0$ such that $H^{\otimes k}$ is very ample, and let $\Psi:X\hookrightarrow\PP^N$ be the associated embedding.  Functoriality gives
\begin{equation*}
T_{\Psi\circ g}(r;\OO_{\PP^N}(1))=kT_g(r;H)+O(1),
\end{equation*}
so the projective characteristic of $\Psi\circ g$ is bounded.  The classical bounded-characteristic criterion in projective space then makes $\Psi\circ g$ constant \cite{GriffithsKing1973Nevanlinna,NoguchiWinkelmann2014Nevanlinna}.  Since $\Psi$ is an embedding, $g$ is constant, contradicting Zariski density in the positive-dimensional variety $X$.
\end{proof}

\section{External results}
\label{app:external-inputs}

For reference we record the results from the literature used above.  The first two come from quaternionic slice regular function theory; the others come from toric geometry and semi-abelian value distribution theory.

\begin{enumerate}[label=\textup{(E\arabic*)}]
\item \textbf{Bisi--Winkelmann zero-divisor criterion.}
For a real-symmetric stem function $F$ and a quaternion $c\in\HH$, the induced slice regular function omits $c$ if and only if
\[
Q_c(z)=\langle F(z)-c,F(z)-c\rangle
\]
is zero-free on $\CC$.  This is the consequence of \cite[Proposition~2.2]{BisiWinkelmann2020QuaternionicPicard} used in Lemma~\ref{lem:zero-divisor}.  The interpretation of $Q_c$ as the stem of the normal function $N(f-c)$ belongs to the standard zero theory of slice regular functions \cite{GentiliStoppato2008Zeros,GentiliStoppatoStruppa2013Regular}.

\item \textbf{Bisi--Winkelmann plane-omission and five-point theorem.}
Bisi--Winkelmann construct entire slice regular functions whose images are complements of real affine $2$-planes; in the notation of their paper this follows from \cite[Proposition~5.1, Proposition~6.1 and the subsequent remark]{BisiWinkelmann2020QuaternionicPicard}.  Their five-point theorem, \cite[Theorem~3.5]{BisiWinkelmann2020QuaternionicPicard}, states that five quaternionic values not contained in any real affine $3$-space cannot be omitted by a nonconstant entire slice regular function.  These two statements are recorded in Theorem~\ref{thm:bw-inputs}.

\item \textbf{Logarithmic Bloch--Ochiai for semi-abelian varieties.}
The required statement is that the Zariski closure of an entire curve in a semi-abelian variety is a translate of a semi-abelian subvariety.  This is the logarithmic Bloch--Ochiai theorem of Noguchi \cite[Main Theorem~(i)]{Noguchi1998HolomorphicCurves}; see also \cite{NoguchiWinkelmannYamanoi2002SMT,DethloffLu2001LogarithmicJetBundles} for related formulations.  For $(\CC^*)^4$, this identifies $M_Q=\overline{Q(\CC)}^{\Zar}$ as a translated algebraic subtorus, as used in Proposition~\ref{prop:positive-torus}.

\item \textbf{NWY level-one truncated Second Main Theorem.}
The divisor case of \cite[Main Theorem~(iii)]{NoguchiWinkelmannYamanoi2008SMTII} is used in the form
\[
T_G(r;L(\overline{\mathcal D}))
\le
N_1(r;G^*\mathcal D)
+
\varepsilon T_G(r;L(\overline{\mathcal D}))
\quad \|_\varepsilon,
\]
for a Zariski-dense entire curve $G$ into a semi-abelian variety and an effective reduced divisor $\mathcal D$.  Here the ambient semi-abelian variety is always a torus.  When a common smooth equivariant refinement is used, the curve still lies in the open torus; the open divisor pullback and the truncated counting function are unchanged, boundary components add no counting term, and functoriality gives the characteristic comparison up to $O(1)$.  The proof of Theorem~\ref{thm:toric-ramification} keeps track of the same line bundle after refinement: the pulled-back NWY line bundle equals the strict Laurent closure plus an effective boundary-exceptional divisor.

\item \textbf{Toric bigness of Laurent divisors.}
The bigness needed in Theorem~\ref{thm:toric-ramification} is supplied by Lemma~\ref{lem:toric-bigness}, the toric specialization of \cite[Proposition~3.9(ii)]{NoguchiWinkelmannYamanoi2008SMTII}.  For a reduced Laurent divisor in a torus, trivial connected stabilizer forces the Newton polytope to be full-dimensional, and the strict closure on a smooth projective toric compactification has big associated line bundle.  Adding the effective boundary-exceptional divisor coming from the NWY compactification preserves bigness.  The toric background is the divisor--polytope correspondence and the bigness criterion for toric divisors; see \cite[\S~3.4]{Fulton1993Toric} and \cite[Proposition~4.3.3 and Ch.~9]{CoxLittleSchenck2011Toric}.

\item \textbf{Big characteristic and algebraic degeneracy.}
Lemma~\ref{lem:appendix-big-characteristic} gives the implication from bounded characteristic for a big line bundle to algebraic degeneracy.  It is used at the end of Theorem~\ref{thm:toric-ramification}; the projective-space input is the classical bounded-characteristic criterion from Nevanlinna theory for projective varieties \cite{GriffithsKing1973Nevanlinna,NoguchiWinkelmann2014Nevanlinna}.
\end{enumerate}

\section*{Acknowledgments}
This work was supported by the National Natural Science Foundation of China (No.~12171448).


\begin{thebibliography}{99}

  \bibitem{BisiWinkelmann2020QuaternionicPicard}
  C. Bisi, J. Winkelmann: \emph{On a quaternionic Picard theorem}. Proc. Amer. Math. Soc., Ser. B \textbf{7}, 106--117 (2020). \url{https://doi.org/10.1090/bproc/54}
  
  \bibitem{BisiWinkelmann2020Harmonicity}
  C. Bisi, J. Winkelmann: \emph{The harmonicity of slice regular functions}. J. Geom. Anal. \textbf{31}, 7773--7811 (2021). \url{https://doi.org/10.1007/s12220-020-00551-7}
  
  \bibitem{Boas1954EntireFunctions}
  R.P. Boas: \emph{Entire Functions}. Pure and Applied Mathematics, vol. 5. Academic Press, New York (1954)
  
  \bibitem{ColomboGentiliSabadiniStruppa2009Extension}
  F. Colombo, G. Gentili, I. Sabadini, D.C. Struppa: \emph{Extension results for slice regular functions of a quaternionic variable}. Adv. Math. \textbf{222}, no.5, 1793--1808 (2009). \url{https://doi.org/10.1016/j.aim.2009.06.015}
  
  \bibitem{ColomboSabadiniStruppa2011FunctionalCalculus}
  F. Colombo, I. Sabadini, D.C. Struppa: \emph{Noncommutative Functional Calculus: Theory and Applications of Slice Hyperholomorphic Functions}. Progress in Mathematics, vol.289. Birkh\"auser/Springer Basel AG, Basel (2011). \url{https://doi.org/10.1007/978-3-0348-0110-2}
  
  \bibitem{CoxLittleSchenck2011Toric}
  D.A. Cox, J.B. Little, H.K. Schenck: \emph{Toric Varieties}. Graduate Studies in Mathematics, vol.124. Amer. Math. Soc., Providence, RI (2011). \url{https://doi.org/10.1090/gsm/124}
  
  \bibitem{Cullen1965Integral}
  C.G. Cullen: \emph{An integral theorem for analytic intrinsic functions on quaternions}. Duke Math. J. \textbf{32}, no.1, 139--148 (1965). \url{https://doi.org/10.1215/S0012-7094-65-03212-6}
  
  \bibitem{DethloffLu2001LogarithmicJetBundles}
  G.-E. Dethloff, S.S.-Y. Lu: \emph{Logarithmic jet bundles and applications}. Osaka J. Math. \textbf{38}, no.1, 185--237 (2001)
  
  \bibitem{Fulton1993Toric}
  W. Fulton: \emph{Introduction to Toric Varieties}. Annals of Mathematics Studies, vol.131. Princeton University Press, Princeton, NJ (1993)
  
  \bibitem{GentiliStoppato2008Zeros}
  G. Gentili, C. Stoppato: \emph{Zeros of regular functions and polynomials of a quaternionic variable}. Michigan Math. J. \textbf{56}, no.3, 655--667 (2008). \url{https://doi.org/10.1307/mmj/1231770366}
  
  \bibitem{GentiliStoppatoStruppa2013Regular}
  G. Gentili, C. Stoppato, D.C. Struppa: \emph{Regular Functions of a Quaternionic Variable}. Springer Monographs in Mathematics. Springer, Heidelberg (2013). \url{https://doi.org/10.1007/978-3-642-33871-7}
  
  \bibitem{GentiliStruppa2006Cullen}
  G. Gentili, D.C. Struppa: \emph{A new approach to {Cullen}-regular functions of a quaternionic variable}. C. R. Math. Acad. Sci. Paris \textbf{342}, no.10, 741--744 (2006). \url{https://doi.org/10.1016/j.crma.2006.03.015}
  
  \bibitem{GentiliStruppa2007NewTheory}
  G. Gentili, D.C. Struppa: \emph{A new theory of regular functions of a quaternionic variable}. Adv. Math. \textbf{216}, no.1, 279--301 (2007). \url{https://doi.org/10.1016/j.aim.2007.05.010}
  
  \bibitem{GhiloniPerotti2011Alternative}
  R. Ghiloni, A. Perotti: \emph{Slice regular functions on real alternative algebras}. Adv. Math. \textbf{226}, no.2, 1662--1691 (2011). \url{https://doi.org/10.1016/j.aim.2010.08.015}
  
  \bibitem{GriffithsKing1973Nevanlinna}
  P.A. Griffiths, J. King: \emph{Nevanlinna theory and holomorphic mappings between algebraic varieties}. Acta Math. \textbf{130}, 145--220 (1973). \url{https://doi.org/10.1007/BF02392265}
  
  \bibitem{Levin1980Zeros}
  B.Ya. Levin: \emph{Distribution of Zeros of Entire Functions}, Revised ed. Translations of Mathematical Monographs, vol.5. Amer. Math. Soc., Providence, RI (1980)
  
  \bibitem{Noguchi1998HolomorphicCurves}
  J. Noguchi, \textit{On holomorphic curves in semi-abelian varieties}, Math. Z. \textbf{228} (1998), no.~4, 713--721.

  \bibitem{Noguchi1981LogDerivatives}
  J. Noguchi, \textit{Lemma on logarithmic derivatives and holomorphic curves in algebraic varieties}, Nagoya Math. J. \textbf{83} (1981), 213--233.

  \bibitem{NoguchiWinkelmann2014Nevanlinna}
  J. Noguchi, J. Winkelmann: \emph{Nevanlinna Theory in Several Complex Variables and Diophantine Approximation}. Grundlehren der mathematischen Wissenschaften, vol.350. Springer, Tokyo (2014). \url{https://doi.org/10.1007/978-4-431-54571-2}
  
  \bibitem{NoguchiWinkelmannYamanoi2002SMT}
  J. Noguchi, J. Winkelmann, K. Yamanoi: \emph{The second main theorem for holomorphic curves into semi-abelian varieties}. Acta Math. \textbf{188}, no.1, 129--161 (2002). \url{https://doi.org/10.1007/BF02392797}
  
  \bibitem{NoguchiWinkelmannYamanoi2007Degeneracy}
  J. Noguchi, J. Winkelmann, K. Yamanoi: \emph{Degeneracy of holomorphic curves into algebraic varieties}. J. Math. Pures Appl. \textbf{88}, no.3, 293--306 (2007). \url{https://doi.org/10.1016/j.matpur.2007.05.002}
  
  \bibitem{NoguchiWinkelmannYamanoi2008SMTII}
  J. Noguchi, J. Winkelmann, K. Yamanoi: \emph{The second main theorem for holomorphic curves into semi-abelian varieties. {II}}. Forum Math. \textbf{20}, no.3, 469--503 (2008). \url{https://doi.org/10.1515/FORUM.2008.024}
  
  \bibitem{Perotti2020Jensen}
  A. Perotti, \textit{A four dimensional Jensen formula}, Riv. Math. Univ. Parma (N.S.) \textbf{11} (2020), no.~1, 139--152.
  
\end{thebibliography}
\end{document}